\documentclass[11pt]{article}
\usepackage{amsfonts}
\usepackage{amsmath}
\usepackage{amsthm}
\usepackage{centernot}
\usepackage{ragged2e}
\usepackage{titling}
\usepackage{tasks}
\usepackage{titlesec}
\usepackage{fancyhdr}
\usepackage{layout}
\usepackage{url}
\usepackage{comment}
\usepackage{indentfirst}
\usepackage[runin]{abstract}

\titleformat{\section}[block]{\normalfont\scshape\filcenter}{\thesection.}{0.5em}{}

\pretitle{\vspace{-2cm}\begin{center}\normalfont\bfseries}
\posttitle{\end{center}}
\preauthor{\begin{center}\small}
\postauthor{\end{center}\vspace{-1cm}}

\newtheorem{theorem}{Theorem}[section]

\newtheorem{lemma}[theorem]{Lemma}

\def\rt{\ensuremath{\rightarrow}}

\begin{document}
\title{A COMBINATORIAL PROOF FOR 132-AVOIDING PERMUTATIONS WITH A UNIQUE LONGEST INCREASING SUBSEQUENCE}
\author{NICHOLAS VAN NIMWEGEN}
\date{}
\maketitle

\begin{abstract}
We provide a simple injective proof that the number of 132-avoiding permutations with a unique longest increasing subsequence is at least as large as the number of 132-avoiding permutations without a unique longest increasing subsequence.
\end{abstract}

\section*{Introduction}

Let $p = p_{1}p_{2}...p_{n}$ be a permutation. An \emph{increasing subsequence} in $p$ is a subset of (not necessarily consecutive) entries $p_{j_{1}} < p_{j_{2}} < ... < p_{j_{k}}$ such that $j_{1} < j_{2} < ... < j_{k}$. 

We say that $p$ has a \emph{unique longest increasing subsequence}, or ULIS, if $p$ has an increasing subsequence that is longer than all other increasing subsequences. For example 34256178 has a ULIS, namely 345678, but 32456178 does not, since 345678 and 245678 are both increasing subsequences of maximal length 6 in $p$.

The number of all permutations of length $n$ that have a unique longest increasing subsequence has proven to be a difficult problem. They are documented for $n \leq 15$ in Sequence A167995 in the Online Encyclopedia of Integer Sequences \cite{oeis}.

An approach taken to the problem in 2020 by Bóna and DeJonge \cite{bona} considered permutations avoiding patterns of length 3. It is well known that the number of such permutations for any length 3 pattern is $C_{n} = {2n \choose n}/(n+1)$, the $n^{\text{th}}$ Catalan number. Adopting their terminology, they showed, as a corollary of a theorem of \cite{kestel}, that if $u_{n}(132)$ is the number of permutations with a ULIS avoiding 132, then $\lim_{n \rt \infty} u_{n}(132)/C_{n} = 0.5$. As a follow-up question, they asked if $u_{n}(132)/C_{n} \geq 0.5$ for all $n$. Finding a simple injective proof of this has been an open problem for four years. In this paper we provide such a proof.

\pagebreak
\section{Preliminary Lemmas}

We begin with two preliminary lemmas, that will prove essential to our approach:

\begin{lemma}

Let $p = p_{1}p_{2}...p_{n}$ avoid 132. Then for all $i \in [n]$, there is a unique longest increasing subsequence in $p$ that begins at $p_{i}$.

\end{lemma}

Note we are not concerned with the subsequence being longest with respect to $p$, but rather with respect to starting at $p_{i}$.

\begin{proof}

Assume otherwise, that there exists $i \in p$ such that $p_{i}a_{2}a_{3}...a_{k}$ and $p_{i}b_{2}b_{3}...b_{k}$ are both maximal length increasing subsequences starting at $p_{i}$. Let $j$ be the smallest natural number such that $a_{j} \neq b_{j}$, and assume without loss of generality $a_{j} < b_{j}$. If $a_{j}$ comes before $b_{j}$, then $p_{i}b_{2}...b_{j-1}a_{j}b_{j}...b_{k}$ is a longer increasing subsequence beginning at $p_{i}$. If instead $b_{j}$ comes before $a_{j}$, then $p_{i}b_{j}a_{j}$ is a 132 pattern. Either way, we get a contradiction. \qedhere

\end{proof}

If the length of the longest increasing subsequence starting at $p_{i}$ is $k$, then we say $p_{i}$ has \emph{rank} $k$. Using this, we can define the rank function $r(p_{i})$ on entries, and $R(p) = r(p_{1})r(p_{2})...r(p_{n})$, where $r(p_{i})$ is the rank of $p_{i}$.

From observation, we can see that $R(p)$ is a sequence of positive integers such that $r(p_{i}) - r(p_{i+1}) \leq 1$ for all $i \in [n-1]$, and that $R(p)$ must end in 1. If $r(p_{i}) - r(p_{i+1}) > 1$, then the longest subsequence starting at $p_{i}$ must not include $p_{i+1}$. Then the second entry of said subsequence, $a$, must be smaller than $p_{i+1}$. Then there must be a 132 pattern formed by $p_{i},p_{i+1},a$.

Let $\mathcal{S}_{n}$ be the set of all sequences of length $n$ satisfying those conditions. Then we have the following lemma:

\begin{lemma}

Let $Av_{n}(132)$ be the set of 132-avoiding permutations of length $n$. Then $R : Av_{n}(132) \rt \mathcal{S}_{n}$ is a bijection.

\end{lemma}

\begin{proof}

First note $|\mathcal{S}_{n}| = C_{n}$, via (6.19u) in \cite{stanley}, although those sequences instead start at 0 and have each entry be at most one greater than the previous. The equivalence can be seen by reversing all sequences and then adding 1 to every entry.

Thus if we can show that $R$ is an injection, we have a bijection. We do this by starting with an element from $\mathcal{S}_{n}$, and finding the inverse of $R$.

First, clearly $n$ must be placed at the leftmost 1. Now for $(n-1)$ we have two cases, either there is a 2 to the left of the 1, or there is not. In the former case, $(n-1)$ must be placed there, as otherwise whatever entry is placed there will either form a 132 pattern (if $(n-1)$ is placed after $n$) or a subsequence of length 3 with $(n-1)$ and $n$ (if $(n-1)$ is placed before $n$). In the latter case, $(n-1)$ must be placed at the leftmost remaining 1.

This process continues similarly. After each $m$ placed at rank $k$, either $m+1$ is placed at the leftmost rank $k+1$ (if it preceeds the placement of $m$) or at the leftmost remaining location of highest rank less than $k+1$ (if no $k+1$ can be found). Any other placement of $m+1$ will give either a longer subsequence in some location, or a 132 pattern.

Since we have a unique inverse, as at each step we only had a single choice, we have an injection, and since $|Av_{n}(132)| = |\mathcal{S}_{n}|$, we have a bijection. \qedhere

\end{proof}

\section{Result}

It becomes clear from the two lemmas that for $p \in Av_{n}(132)$, $p$ has a ULIS iff $R(p)$ has a maximum value. To show that $u_{n}(132)/C_{n} \geq 0.5$, we show an injection from permutations without a ULIS to permutations with one.

Let $\mathcal{T}_{n}$ be the set of elements of $\mathcal{S}_{n}$ without a unique maximum element, and $\mathcal{T'}_{n}$ be $\mathcal{S}_{n} - \mathcal{T}_{n}$. We define the function $f: \mathcal{T}_{n} \rt \mathcal{T'}_{n}$, which we will later claim is an injection, as follows:

Let $T = t_{1}t_{2}...t_{n} \in \mathcal{T}_{n}$, with maximal element $k$. Let $i,j$ be the final two positions of $T$ such that $t_{i} = t_{j} = k$, with $i < j$. We increase all values on the range $[i,j)$ by one.

This clearly maps from $\mathcal{T}_{n} \rt \mathcal{T'}_{n}$, as by choice of $i,j$, only one value on the range $[i,j)$ has maximal value, and that value becomes a maximum of $k+1$.

Using this, we can define an injection $g$ from permutations of length $n$ without a ULIS to permutations with a ULIS:

\begin{theorem}

Let $U_{n}(132)$ be the set of 132-avoiding permutations with a ULIS, and $V_{n}(132)$ be the set of 132 avoiding permutations without a ULIS. Then $g: V_{n}(132) \rt U_{n}(132)$ defined by $g = R^{-1} \circ f \circ R$ is an injection. Thus, $|U_{n}(132)| \geq |V_{n}(132)|$.

\end{theorem}

\begin{proof}

First, note that $g$ clearly maps into $U_{n}(132)$, as by choice of $i,j$, only one element on $[i,j)$ is of maximal rank $k$, and following the transformation we have only one element of rank $k+1$.

As $R,R^{-1}$ are bijections, it suffices to show that the middle step is an injection. For two elements of $\mathcal{T}_{n}$ to map into the same element of $\mathcal{T'}_{n}$, they would need to have the same $k,i,j$, and values of ranks at all positions, but then the two elements are clearly the same. Thus, we have an injection. \qedhere

\end{proof}

\pagebreak

\end{document}